\documentclass[12pt]{article}

\usepackage{amsmath,amsfonts,amsthm,amssymb, fullpage}
\usepackage{hyperref}

\newcommand{\K}{\mathbb{R}}
\newcommand{\perm}[1]{~\mathrm{per}\left[ #1 \right]}

\newcommand{\one}{\vec{1}}

\newcommand{\trans}{\intercal}
\renewcommand{\hat}{\widehat}

\newcommand{\AND}{\quad\text{and}\quad}

\newcommand{\mydet}[1]{\det \left[ #1 \right]}

\newcommand{\Qmat}[2]{{Q_{#1}^{#2}}}
\newcommand{\vx}{\vec{x}}
\newcommand{\minor}[1]{\left|#1\right|}

\newcommand{\vu}{\vec{u}}
\newcommand{\vv}{\vec{v}}
\newcommand{\vw}{\vec{w}}

\newtheorem{theorem}{Theorem}
\newtheorem{corollary}[theorem]{Corollary}
\newtheorem{lemma}[theorem]{Lemma}
\newtheorem{conjecture}[theorem]{Conjecture}

\begin{document}

\title{A Determinantal Identity for the Permanent of a Rank 2 Matrix}
\markright{Determinantal Identity for Rank 2 Permanents}
\author{Adam W. Marcus \\
\'Ecole Polytechnique F\'ed\'erale de Lausanne 
}
\maketitle

\begin{abstract}
We prove an identity relating the permanent of a rank $2$ matrix and the 
determinants of its Hadamard powers.
When viewed in the right way, the resulting formula looks strikingly 
similar to an identity of Carlitz and Levine, suggesting the 
possibility that these are actually special cases of some more general identity 
(or class of identities) connecting permanents and determinants.
The proof combines some basic facts from the theory of symmetric functions with 
an application of a famous theorem of Binet and Cauchy in linear algebra.


\end{abstract}

\section{Introduction}\label{sec:intro}

The relationship between the {\em determinant} function, which maps a square 
matrix $A$ to
\[
\mydet{A} = \sum_{\sigma \in S_n} (-1)^{|\sigma|} \prod_{i=1}^n A(i, \sigma(i)),
\]
and the {\em permanent} function, which maps a square matrix $B$ to
\[
\perm{B} = \sum_{\sigma \in S_n}\prod_{i=1}^n B(i, \sigma(i)),
\]
is an important open problem in complexity theory.
Despite having similar forms, the computation of the determinant can be done 
efficiently (due to, among other things, its predictable behavior with respect 
to Gaussian 
elimination) while the computation of the permanent is thought to be 
significantly harder.
One of the more successful approaches to relating these two functions, known as 
{\em geometric complexity theory}, is to look for formulas of the type 
\begin{equation}\label{eq:perm-det}
\perm{M} = \mydet{M'},
\end{equation}
where $M$ is a square matrix and $M'$ is a (typically much larger) square 
matrix formed from affine combinations of the entries of $M$.
(The size of the matrix $M'$ that is needed for 
\eqref{eq:perm-det} to hold can then be related to other measures of complexity 
--- we refer the reader to \cite{burg} for more details.) 

For matrices of a specific form, however, it is sometimes possible to find a 
formula that expresses the permanent of a matrix as the determinant of 
other matrices {\em of the same size}. 
One notable example of this is a result of Borchardt \cite{borchardt} that was 
later generalized by Cayley \cite{cayley} and then further generalized by 
Carlitz and Levine \cite{carlitz}.
For a matrix $M \in \K^{n \times n}$ with entries $m(i, j)$ and integer $p$, 
let $M_p$ be the matrix with 
\[
M_p(i, j) = m(i, j)^p.
\]
In this notation, the main result of \cite{carlitz} is the following theorem:
\begin{theorem}[Carlitz--Levine]\label{thm:cl}
For a rank 2 matrix $M$ with no zero entries, 
\begin{equation}
\label{eq:them}
\mydet{M_{-2}} = \mydet{M_{-1}} \perm{M_{-1}}.
\end{equation}
\end{theorem}
The proof in \cite{carlitz} is elementary, using little more than the 
definitions and some facts concerning the cycle structure of permutations.
The goal of this article is to prove a formula that has an intriguingly 
similar form to \eqref{eq:them} but for seemingly quite different reasons. 
Our main result, Theorem~\ref{thm:main}, states that for matrices $M$ with rank 
at most $2$,
\begin{equation}\label{eq:us}
(n!)^2 \mydet{M_{n}} = (n^n) \mydet{M_{n-1}}\perm{M_1}.
\end{equation}
The proof will use a combination of tools from the theory of polynomials and a 
theorem of Binet and Cauchy on the minors of a product of matrices.

\subsection{Motivation.}\label{sec:motivation}

While the restriction to matrices of rank at most 2 may seem overly simplistic, 
it should be noted that the permanents of such matrices appear naturally in the 
context of binary operations that are symmetric in both arguments (or BOSBAs).
The author came across them, for example, during an investigation of the 
characteristic polynomials of random matrices.
More specifically, let $A, B \in \mathbb{R}^{n \times n}$ be Hermitian 
matrices, and 
consider 
the (random) polynomial
\[
r(x) = \mydet{x I - A - Q^T B Q},
\]
where $Q$ is an orthogonal matrix drawn uniformly from $\mathcal{O}_n$ (the 
group of $n \times n$ orthogonal matrices) with respect to the Haar measure.

It is not hard to see that the value of $r(x)$ is independent of the 
eigenvectors of $B$ and therefore must be a symmetric function of the 
eigenvalues of $B$.
Conjugating by $Q$, one can see that the same is true for $A$ as well.
That is, when viewed as a function of the eigenvalues of $A$ (as the first 
argument) and the eigenvalues of $B$ (as the second argument), $r(x)$ is a 
BOSBA. 
Furthermore, it can be shown that the expected value of this polynomial is the 
permanent of a rank 2 matrix:
\[
\mathbb{E}_{Q} \left\{  \mydet{x I - A - Q^T B Q} \right\} = \frac{1}{n!}\perm{ 
\{ x - 
\lambda_i(A) - \lambda_j(B) \}_{i, j=1}^n },
\]
where $\lambda_i(A)$ denotes the eigenvalues of $A$ (and similarly for $B$).

Expected characteristic polynomials of this type play an important role in the 
field of {\em finite free probability} which has contributed to a 
number of recent theoretical \cite{gorin, if3} and algorithmic \cite{if4, 
xie_xu} advances.
Since permanents are, by nature, often more difficult to work with than 
determinants (they are not multiplicative, for example), the hope is that a 
formula like 
\eqref{eq:us} could have uses (though, 
admittedly, the author has not found any yet --- see 
Section~\ref{sec:conclusion}).

\section{Preliminaries}

We will use the customary notation that $[n] = \{ 1, 2, \dots, n \}$ and that 
$\binom{[n]}{k}$ denotes the collection of subsets of $[n]$ size $k$.
For a permutation $\sigma$, we write $|\sigma|$ to denote the number of cycles 
in its cycle decomposition.
For a vector $\vx \in \K^n$, a permutation $\sigma \in S_n$, and a set $S 
\subseteq [n]$, we will write
\[
\vx^S := \prod_{i \in S} x_i
\AND
\sigma(S) = \{ \sigma(i) : i \in S \}.
\]
In particular, given a matrix $M$ and sets $I, J$, we will write $M(I, J)$ to 
denote the submatrix of $M$ formed by the rows in $I$ and columns in $J$.

\subsection{Symmetric and alternating polynomials}

A polynomial $p \in \K[x_1, \dots, x_n]$ is said to be {\em 
symmetric} if 
\[
p(x_1, \dots x_i, x_{i+1}, \dots,  x_n) = p(x_1, \dots x_{i+1}, x_{i}, \dots,  
x_n) 
\]
and {\em alternating} if 
\[
p(x_1, \dots x_i, x_{i+1}, \dots,  x_n) = -p(x_1, \dots x_{i+1}, x_{i}, \dots,  
x_n) 
\]
for all transpositions $(i, i+1)$.
Since the set of transpositions generates the symmetric group, equivalent 
definitions are (for symmetric polynomials)
\[
p(x_1, \dots,  x_n) = p(x_{\pi(1)}, \dots,  x_{\pi(n)}) 
\]
and (for alternating polynomials)
\begin{equation}
\label{eq:alternate}
p(x_1, \dots,  x_n) = (-1)^{|\pi|} p(x_{\pi(1)}, \dots,  x_{\pi(n)})
\end{equation}
for all $\pi \in S_n$.
In particular, (\ref{eq:alternate}) implies that any alternating polynomial $p$ 
must be $0$ whenever $x_i = x_j$ for some $i \neq j$.

Examples of symmetric polynomials are the {\em elementary symmetric polynomials}
\[
e_k(x_1, \dots, x_n) = 
\begin{cases}
1 & \text{for $k = 0$,}\\
\sum\limits_{S \in \binom{[n]}{k}} \vx^S & \text{for $1 \leq k \leq n$,}\\
0 & \text{otherwise}
\end{cases}
\]
and the {\em power sum polynomials}
\[
p_k(x_1, \dots, x_n) = \sum_{i=1}^n x_i^k.
\]
One example of an alternating polynomial is the {\em Vandermonde polynomial}
\begin{equation}
\label{eq:vand}
\Delta(x_1, \dots, x_n) = \prod_{i < j} (x_j - x_i). 
\end{equation}
Furthermore, it is easy to see that the Vandermonde polynomial is an essential 
part of any alternating polynomial:
\begin{lemma}
\label{lem:alternate}
For all alternating polynomials $f(x_1, \dots, x_n)$, there exists a symmetric 
polynomial $t(x_1, \dots, x_n)$ such that
\[
f(x_1, \dots, x_n) = \Delta(x_1, \dots, x_n)t(x_1, \dots, x_n).
\]
\end{lemma}
\begin{proof}
For distinct $y_2, \dots, y_n \in \K$, (\ref{eq:alternate}) implies that the 
univariate polynomial
\[
g(x) = f(x, y_2, \dots, y_n) \in \K[x]
\]
satisfies $g(y_k) = 0$ for each $k = 2, \dots, n$.
Hence $(x - y_k)$ must be a factor of $g$ and so $(x_1 - x_k)$ must be a factor 
of $f$.
Since this is true for all $k$ and all $i$ (not just $i=1$), every polynomial 
of the form $(x_i - x_k)$ must be a factor of $f$, and so 
\begin{equation}\label{eq:decomposition}
f(x_1, \dots, x_n) = \Delta(x_1, \dots, x_n)t(x_1, \dots, x_n)
\end{equation}
for some polynomial $t$. 
It remains to show that $t$ is symmetric.
To do so, for each $\sigma \in S_n$, we can apply \eqref{eq:decomposition} to 
\eqref{eq:alternate} to get
\[
\Delta(x_1, \dots, x_n)t(x_1, \dots, x_n) = 
(-1)^{|\sigma|}\Delta(x_{\sigma(1)}, \dots, x_{\sigma(n)})t(x_{\sigma(1)}, 
\dots, x_{\sigma(n)}).
\]
Since $\Delta$ is, itself, an alternating polynomial, we have
\[
\Delta(x_1, \dots, x_n) = (-1)^{|\sigma|}\Delta(x_{\sigma(1)}, \dots, 
x_{\sigma(n)})
\]
and so canceling both sides results in
\[
t(x_1, \dots, x_n) = t(x_{\sigma(1)}, \dots, x_{\sigma(n)})
\]
as needed.
\end{proof}

\subsection{Linear algebra}

The only tool we will need from linear algebra is a famous theorem of Binet and 
Cauchy regarding the minors of products of matrices.
The {\em $(I, J)$-minor} of a matrix $M$ is defined as
\[
\minor{M}_{I, J}
= 
\mydet{ M(I, J) }.
\]
The following theorem is attributed to Binet and Cauchy (independently) 
\cite{cb}.

\begin{theorem}[Cauchy--Binet]\label{thm:cb}
Let $m, n, p$ and $k$ be positive integers for which $k \leq \min\{m, n, p\}$.
Then for all $m \times n$ matrices $A$ and $n \times p$ matrices $B$ and all 
sets $I, J$ with $|I| = |J| = k$, we have
\begin{equation}\label{eq:cb}
\minor{ A B }_{I, J} = \sum_{K \in \binom{[n]}{k}}  
\minor{A}_{I, K}
\minor{B}_{K, J}
\end{equation}
\end{theorem}

For those unfamiliar with Theorem~\ref{thm:cb}, we hope to convey some 
appreciation for it by pointing out that \eqref{eq:cb} simultaneously 
generalizes two fundamental formulas from linear algebra that (a priori) have 
no 
obvious relation to each other: the formula for matrix multiplication (the case 
when $k = 1$), and the product formula for determinants (the case when 
$m=n=p=k$).

\section{The main theorem}

Our approach will be to first prove (\ref{eq:us}) in a special case 
(Corollary~\ref{cor:main}) and to then extend that result to the full theorem.
We start by finding an expansion for the permanent of certain matrices in terms 
of the elementary symmetric polynomials:
\begin{lemma}\label{lem:perm}
For vectors $\vu, \vv \in \K^n$, if $A$ is the $n \times n$ matrix with 
\[
A(i, j) = 1 + u_i v_j,
\]
then 
\[
\perm{A} 
= \sum_{k=0}^n k! (n-k)! e_k(\vu)e_k(\vv).
\]
\end{lemma}
\begin{proof}
By definition, we have
\[
\perm{A} 
= \sum_{\sigma \in S_n} \prod_{i=1}^n (1 + u_iv_{\sigma(i)}) 
\]
where for each $\sigma$, we have
\[
\prod_{i=1}^n (1 + u_iv_{\sigma(i)}) 
= \sum_{S \subseteq [n]} \vu^S \vv^{\,\sigma(S)}.
\]
For fixed $S$ with $|S| = k$, as $\sigma$ ranges over all permutations, 
$\sigma(S)$ will range over all sets $T \in \binom{[n]}{k}$ and any $\sigma'$ 
for which $\sigma'(S) = \sigma(S)$ will give the same term.
As there are a total of $k!(n-k)!$ such permutations, we have
\[
\perm{A} 
= \sum_{k=0}^n k! (n-k)! \sum_{S \in \binom{[n]}{k}}\sum_{T \in \binom{[n]}{k}} 
\vu^S \vv^{T}
= \sum_{k=0}^n k! (n-k)! e_k(\vu)e_k(\vv)
\]
as claimed.
\end{proof}

For a vector $\vx \in \K^n$, let $\{ Q_{\vx}^k \}_{k=0}^n$ be the 
collection of $n \times n$ matrices with entries
\begin{equation}\label{eq:jac}
\Qmat{\vx}{k}(i, j) = 
\begin{cases}
x_j^{i} & \text{if $i > k$,} \\
x_j^{i-1} & \text{otherwise}.
\end{cases}
\end{equation}
For example, when $\vx = (a, b, c)$, we have
\[
\Qmat{\vx}{0} =
\begin{bmatrix} 
a & b & c \\
a^2 & b^2 & c^2 \\
a^3 & b^3 & c^3
\end{bmatrix}
,~
\Qmat{\vx}{1} =
\begin{bmatrix} 
1 & 1 & 1 \\
a^2 & b^2 & c^2 \\
a^3 & b^3 & c^3
\end{bmatrix}
,~\dots~,~ 
\Qmat{\vx}{3} =
\begin{bmatrix} 
1 & 1 & 1 \\
a & b & c \\
a^2 & b^2 & c^2
\end{bmatrix}.
\]

\begin{lemma}\label{lem:Q}
For any vector $\vx \in \K^n$, we have 
\[
\mydet{ \Qmat{\vx}{k} } 
= 
e_{n-k}(\vx) \Delta(\vx).
\]
\end{lemma}
\begin{proof}
It should be clear from the properties of determinants that each function 
\[
f_k(\vx) := \mydet{ \Qmat{\vx}{k} }
\]
is an alternating polynomial.
Hence, by Lemma~\ref{lem:alternate}, we have that 
\[
f_k(\vx) = t_k(\vx) \Delta(\vx)
\]
for some symmetric function $t_k$, so it remains to show that $t_k = e_k$.

We start by considering degrees.
Note that each $f_k(\vx)$ is a homogeneous polynomial of degree 
$\binom{n+1}{2} - k$ and that $\Delta(\vx)$ is a homogeneous 
polynomial of degree $\binom{n}{2}$.
Hence $t_k(\vx)$ must be a homogeneous polynomial of degree $n-k$.
Now note that the degree of each $x_i$ in $f_k(\vx)$ is at 
most $n$, whereas the the degree of each $x_i$ in $\Delta(\vx)$ is 
$n-1$.
Hence the degree of each $x_i$ in $t_k(\vx)$ is at most $1$.

The only symmetric polynomials that satisfy these constraints are multiples of 
the elementary symmetric polynomials, and so we must have
\[
f_k(\vx) = w_k e_{n-k}(\vx) \Delta(\vx)
\]
for some constant $w_k$.
However, one can check that we must have $w_k = 1$ by looking at the monomial 
formed by the product of the diagonal entries in $\Qmat{\vx}{k}$ and 
seeing that it is positive.
\end{proof}

We will now use Lemma~\ref{lem:Q} to prove two corollaries.
The first corollary appears in the solution to Problem 293(a) in \cite{mir} 
(and we suspect 
the authors were aware of Lemma~\ref{lem:Q} though they did not explicitly 
state it).

\begin{corollary}\label{cor:fn-1}
For vectors $\vu, \vv \in \K^n$, if $B$ is the $n \times n$ 
matrix with entries
\[
B(i,j)  = (1 + u_i v_j)^{n-1},
\]
then 
\[
\mydet{B} = \Delta(\vu)\Delta(\vv) \left(\prod_{j=0}^{n-1} \binom{n-1}{j} 
\right).
\]
\end{corollary}
\begin{proof}
Let $Y$ be the $n \times n$ diagonal matrix with entries $Y(i, i) = 
\binom{n-1}{i-1}$.
It is easy to check that 
\[
B = \Qmat{\vu}{n}^{\trans} Y \Qmat{\vv}{n}
\]
and so by the product rule for the determinant, we have
\[
\mydet{B} 
= \mydet{\Qmat{\vu}{n}}\mydet{Y}\mydet{ \Qmat{\vv}{n}}.
\]
Lemma~\ref{lem:Q} then completes the proof.
\end{proof}

The second corollary is similar in spirit to the first one, but requires the 
added machinery of Theorem~\ref{thm:cb}.

\begin{corollary}\label{cor:fn}
For vectors $\vu, \vv \in \K^n$, if $C$ is the $n \times n$ matrix with 
entries
\[
C(i,j)  = (1 + u_i v_j)^{n},
\]
then 
\[
\mydet{C} = \Delta(\vu)\Delta(\vv) \left(\prod_{j=0}^{n} \binom{n}{j} 
\right) \sum_{k=0}^n \frac{e_{k}(\vu) e_{k}(\vv)}{\binom{n}{k}} .
\]
\end{corollary}
\begin{proof}
Let $\hat{Q}_{\vx}$ be the $n \times (n+1)$ matrix with entries
\[
\hat{Q}_{\vx}(i, j) = x_i^{j-1}
\] 
and let $Z$ be the $(n+1) \times (n+1)$ diagonal matrix with $Z(i, i) = 
\binom{n}{i-1}$.
Then we have $C = \hat{Q}_{\vu} Z \hat{Q}_{\vv}^\trans$
and so we can use Cauchy--Binet (twice) to expand $\mydet{C}$:
\begin{align}
\mydet{C} 
&= \minor{\hat{Q}_{\vu} Z \hat{Q}_{\vv}^\trans}_{[n], [n]}
= \sum_{K \in \binom{[n+1]}{n}} 
\minor{\hat{Q}_{\vu}}_{[n], K} \minor{Z \hat{Q}_{\vv}^\trans}_{K, [n]}
\notag \\&= \sum_{K, L \in \binom{[n+1]}{n}} 
\minor{\hat{Q}_{\vu}}_{[n], K} \minor{Z}_{K, L} 
\minor{\hat{Q}_{\vv}^\trans}_{L, [n]}. \label{eq:expand}
\end{align}
To simplify \eqref{eq:expand}, we note that since $Z$ is diagonal, 
$\minor{Z}_{K, L} = 0$ unless 
$K = L$.
Furthermore, note that the complement of each set $K$ contains 
a single element.
When that element is $p \in [n+1]$, we have
\[
\minor{\hat{Q}_{\vx}}_{[n], K} = \minor{\Qmat{\vx}{p}}
\AND 
\minor{Z}_{K, K} = \frac{\mydet{Z}}{\binom{n}{p-1}}.
\]
The result then follows from Lemma~\ref{lem:Q}.

\end{proof}

Putting the three corollaries together gives us the following:

\begin{corollary}\label{cor:main}
For vectors $\vu, \vv \in \K^n$, if $A, B, C$ are $n \times n$ matrices 
with entries
\[
A(i, j) = 1 + u_i v_j,
\quad
B(i, j) = (1 + u_i v_j)^{n-1},
\AND
C(i, j) = (1 + u_i v_j)^{n},
\]
then
\[
(n!)^2 \mydet{C} = (n^n) \mydet{B} \perm{A}.
\]
\end{corollary}
\begin{proof}
Combining Lemma~\ref{lem:perm} and 
Corollary~\ref{cor:fn} gives 
\[
n! \mydet{C} = \Delta(\vu)\Delta(\vv) \left(\prod_{j=0}^{n} \binom{n}{j} 
\right) \perm{A},
\]
where we have
\[
\prod_{j=0}^n \binom{n}{j}
= \prod_{j=1}^n \frac{n}{j}\binom{n-1}{j-1}
= \frac{n^n}{n!}\prod_{j=1}^n \binom{n-1}{j-1}
= \frac{n^n}{n!} \prod_{j=0}^{n-1} \binom{n-1}{j}.
\]
Plugging in Corollary~\ref{cor:fn-1} gives the result.
\end{proof}

Finally, we extend Corollary~\ref{cor:main} to the case of any rank $2$ matrix.
\begin{theorem}\label{thm:main}
Let $X \in \K^{n \times n}$ be any rank $2$ matrix and let $X_{n-1}, X_n \in \K^{n \times n}$ be the matrices with 
\[
X_{n-1}(i, j) = X(i, j)^{n-1}
\AND
X_{n}(i, j) = X(i, j)^n.
\]
Then
\[
(n!)^2 \mydet{X_n} = (n^n) \mydet{X_{n-1}} \perm{X}.
\]
\end{theorem}
\begin{proof}
First note that if we let $\one \in \K^n$ denote the vector with $\one(k) = 1$ 
for all $k$, then Corollary~\ref{cor:main} proves the theorem in the case that
\[
X = \vu \otimes \vv + \one \otimes \one.
\]
Now let $Y =  \vu \otimes \vv + \vw \otimes \vx$ for general $\vw, \vx$.
The expansion of $\mydet{Y_n}$ in terms of monomials has the form
\begin{equation}
\label{eq:Y1}
\mydet{Y_n} 
= \sum_{i_1, \dots, i_n, j_1, \dots, j_n} c_{i_1, \dots, i_n, j_1, \dots, j_n} 
\prod_k u_k^{i_k}w_k^{n - i_k}v_k^{j_k}x_k^{n - j_k}
\end{equation}
where the $c_{i_1, \dots, i_n, j_1, \dots, j_n}$ are constants.
Similarly, $\mydet{Y_{n-1}} \perm{Y}$ has an expansion
\begin{equation}
\label{eq:Y2}
\mydet{Y_{n-1}} \perm{Y} 
= \sum_{i_1, \dots, i_n, j_1, \dots, j_n} \hat{c}_{i_1, \dots, i_n, j_1, \dots, 
j_n} 
\prod_k u_k^{i_k}w_k^{n - i_k}v_k^{j_k}x_k^{n - j_k}
\end{equation}
for some constants $\hat{c}_{i_1, \dots, i_n, j_1, \dots, 
j_n}$. 
Plugging in $\vx = \one$ and $\vw = \one$, however, does not cause any of the 
coefficients to combine.
That is,
\[
\mydet{X_n} 
= \sum_{i_1, \dots, i_n, j_1, \dots, j_n} c_{i_1, \dots, i_n, j_1, \dots, j_n} 
\prod_k u_k^{i_k}v_k^{j_k}
\]
and
\[
\mydet{X_{n-1}} \perm{X} 
= \sum_{i_1, \dots, i_n, j_1, \dots, j_n} \hat{c}_{i_1, \dots, i_n, j_1, \dots, 
j_n} 
\prod_k u_k^{i_k}v_k^{j_k},
\]
and so by Corollary~\ref{cor:main}, we have
\[
(n!)^2 c_{i_1, \dots, i_n, j_1, \dots, j_n} = (n^n) \hat{c}_{i_1, \dots, i_n, 
j_1, 
\dots, j_n}
\]
for all indices $i_1, \dots, i_n$ and $j_1, \dots, j_n$.
Plugging this into (\ref{eq:Y1}) and (\ref{eq:Y2}) implies equality for $Y$.
\end{proof}

\section{Conclusion}\label{sec:conclusion}

It is worth noting that Theorem~\ref{thm:main} does not seem to be useful 
algorithmically.
For the purpose of computing permanents, it tends to be slower than the 
algorithm of Barvinok \cite{barvinok} and also runs into stability issues 
whenever the two determinants approach $0$ (in particular, when the original 
matrix $X$ is actually rank 1).

Returning to the motivation discussed in Section~\ref{sec:intro}, the
author discovered Theorem~\ref{thm:main} in a (failed) attempt to prove the 
following conjecture, which would have useful implications in the 
field of finite free probability:
\begin{conjecture}
For any matrix $T$ with rank at most $2$,  
\[
\perm{
\begin{matrix}
T & T \\
T & T
\end{matrix}
}
\leq 
\binom{2n}{n}
\perm{T}^2.
\]
\end{conjecture}
\noindent A proof is known in the case that the entries of $T$ are all 
positive, but the general case seems harder.

\subsection{Further research}

There are two obvious directions for extending Theorem~\ref{thm:main}, 
both of which would be quite 
interesting.  
First, the structural similarity between \eqref{eq:them} and \eqref{eq:us} 
suggests that a larger class of formulas of the type 
\[
c(n) \det[M_i] = \mydet{M_j} \perm{M_k}
\]
could exist for matrices $M \in \mathbb{R}^{n\times n}$ of a certain type (in 
our case, rank 2).
Second, one can ask if there is a determinantal formula similar to 
\eqref{eq:us} that is capable of computing the permanent of a rank $3$ matrix.
In this regard, it is worth mentioning that the author's original proof of 
Theorem~\ref{thm:main} used a formula of Jacobi that relates the determinants 
of the $\Qmat{\vx}{k}$ matrices defined in \eqref{eq:jac} to {\em Schur 
polynomials} \cite{jacobi}.
The author opted for the current presentation, which was suggested by an 
anonymous referee, due to its elegance.  
For the purposes of extension, however, we mention the connection to Schur 
polynomials as a possible approach.

A final (but far more speculative) research direction lies in the fundamental 
relationship between the permanent and determinant functions.
In particular, one can ask whether different models for geometric complexity 
theory could lead to new and interesting results. 
Equations \eqref{eq:them} and \eqref{eq:us} suggest two possible 
generalizations.
First, one could attempt to satisfy \eqref{eq:perm-det} with a matrix $M'$ 
whose entries are {\em polynomials} in the 
entries of $M$ (instead of just affine combinations).
Second, one could try to replace \eqref{eq:perm-det} with an equation of the 
form 
\[
\mydet{M''}\perm{M} = \mydet{M'},
\]
where both $M''$ and $M'$ are affine (or polynomial) combinations of the 
entries of $M$.
It is unclear what the immediate implications of either generalization would 
be, but given the central importance of complexity theory in computer science, 
one would expect that any new results of this type would be quite interesting.



\end{document}